\documentclass{amsart}
\usepackage{amsmath, amsfonts, amssymb, amsthm, enumitem, xcolor, hyperref} %
\usepackage{theoremref}
\usepackage{comment}
\usepackage{graphicx}
\newtheorem{theorem}{Theorem}

\newtheorem{lemma}[theorem]{Lemma}

\newtheorem{proposition}[theorem]{Proposition}
\newtheorem{remark}[theorem]{Remark}

\newcommand{\N}{\mathbb{N}}
\newcommand{\R}{\mathbb{R}}
\newcommand{\E}{\mathbf{E}}
\newcommand{\f}{\frac}
\renewcommand{\P}{\mathbf{P}}

\newcommand{\config}{\sigma}

\NewDocumentCommand{\weakconfig}{o}{%
  \config_{\Weak %
  \IfValueTF{#1}{, #1}{}%
  }
}
\NewDocumentCommand{\weakodom}{o}{%
  \odom_{\Weak %
  \IfValueTF{#1}{, #1}{}%
  }
}

\NewDocumentCommand{\Returns}{o}{%
  \mathsf{Returns}%
  \IfValueTF{#1}{(#1)}{}%
}

\newcommand{\Weak}{\mathsf{W}}

\newcommand{\StabSymb}{\mathsf{S}}
\newcommand{\OdomSymb}{\mathsf{O}}
\NewDocumentCommand{\SpacedOp}{ O{} m m }{%
  {#2#3}%
  \IfValueTF{#1}{^{#1}\!}{}%
}
\NewDocumentCommand{\Stab}{o}{\SpacedOp[#1]{}{ \StabSymb }}
\NewDocumentCommand{\WeakStab}{o}{\SpacedOp[#1]{\Weak}{ \StabSymb }}
\NewDocumentCommand{\Odom}{o}{\SpacedOp[#1]{}{ \OdomSymb }}
\NewDocumentCommand{\WeakOdom}{o}{\SpacedOp[#1]{\Weak}{ \OdomSymb }}
\newcommand{\odom}{m}

\newcommand{\Jumps}{\mathfrak{J}}

\DeclareMathOperator{\Ber}{Bernoulli}
\DeclareMathOperator{\var}{Var}
\DeclareMathOperator{\cov}{Cov}
\DeclareMathOperator{\Pois}{Pois}
\DeclareMathOperator{\Geo}{Geo}

\newcommand{\Bin}{\mathrm{Binomial}}
\newcommand{\Exp}{\mathrm{Exp}}

\newcommand{\Tail}{\bar{F}}

\newcommand{\ExcTime}{L}
\renewcommand{\epsilon}{\varepsilon}

\title[Scaling limit and density conjecture for activated random walk]{Scaling limit and density conjecture for activated random walk on the complete graph}

\author{Matthew Junge} \address{Matthew Junge, Department of Mathematics, Baruch College, City University of New York}
	\email{\texttt{matthew.junge@baruch.cuny.edu}}
\author{Harley Kaufman}\address{Harley Kaufman, Department of Mathematics, Baruch College, City University of New York}
	\email{\texttt{Harkauf100@gmail.com} } 
\author{Josh Meisel}\address{Josh Meisel, Department of Mathematics, Graduate Center, City University of New York}
	\email{\texttt{jmeisel@gradcenter.cuny.edu}}

\begin{document}

\begin{abstract}
    We study driven-dissipative activated random walk with sleep probability $p$ on an $n$-vertex complete graph with a sink that traps jumping particles with probability $q_n$. We show that the number of sleeping particles $S_n$ left by the stationary distribution has a Gumbel scaling limit for $\exp(-n^{1/3}) \ll q_n \ll n^{-1/2}$. The particular scaling implies that $S_n$ is hyperuniform and thus the stationary configuration law has negative correlations and is not a product measure. We also prove that $S_n/n$ converges to $p$ if and only if $q_n = e^{-o(n)}$, and that, when $q_n=0$, the number of jumps to stabilization undergoes a phase transition at density $p$.
\end{abstract}

\maketitle

\section{Introduction}

Activated random walk (ARW) is a promising sandpile model of Bak, Tang, and Wiesenfeld's theory of self-organized criticality which posits that gradual accumulation of tension that is released in bursts causes the critical-like behavior observed in diverse natural systems such as snow slopes, tectonic plates, and star surfaces \cite{bak1987soc, dickman2010activated}. 
In ARW, active particles perform independent continuous-time simple random walks on a graph at {jump rate} $1$, and fall asleep at {sleep rate} $\lambda > 0$. Sleeping particles do not move, but become active when visited by another particle.  The probability that a particle tries to fall asleep on a given step is $p := \lambda/(1+\lambda).$ 

Jara\'i, M\"onch, and Taggi initiated the study of \emph{driven-dissipative} ARW on the complete graph $K_{n+1}$ on $n+1$ vertices with vertex $n+1$ designated as the \emph{sink} \cite{jarai2024critical}. This is a Markov chain with a state space of \emph{configurations} in $\{0,\mathfrak s\}^{[n]}$ with $[n] := \{1,2,\hdots, n\}$. Here state $0$ represents an empty site and state $\mathfrak s$ represents a site with one sleeping particle. At each time step $t > 0$, an active particle is added to an independent uniformly chosen vertex $v_t \in [n]$, then the system is stabilized to an all-sleeping configuration. Any particles that jump to the sink during the stabilization are immediately removed. 
A consequence of Levine and Liang's work \cite{levine2021exact} is that this Markov chain has a unique stationary distribution that can be exactly sampled by stabilizing the configuration $1_{[n]}$ that contains one active particle per site. Let $S_n'$ be the number of sleeping particles in a configuration sampled in this manner.

The main theorem of \cite{jarai2024critical}  is that for any $\epsilon >0$ as $n \to \infty$  
\begin{align}
\P(|S_n' - (p n + \alpha_n)|> \epsilon \alpha_n) \to 0, \label{eq:1.2}
\end{align} 
with $\alpha_n := \sqrt{ p (1-p) n \log n}$.
This implies that $S_n' / n \to p$ in probability and pins down the exact first-order shift $\alpha_n + o(\alpha_n)$ above $pn$.  The appearance of $p$ as the mean-field critical density was predicted in \cite{dickman2010activated} and has been observed in other settings such as totally directed walks \cite{HS04} and the high-dimensional limit on $\mathbb Z^d$ \cite{junge2025mean}. 

Although the critical value is arguably the most central theme in ARW research \cite{rollaSurvey}, it is known exactly only in settings where it equals $p$. For mean-field and directed ARW the stationary distribution is a product measure, which lacks the desired complexity of a model of self-organized criticality. Our results demonstrate that the complete graph is a richer setting.

\subsection{Results}
 We study ARW on the graph $\bar K_{n+1}$ obtained from $K_{n+1}$ by adding a self-loop to each vertex in $[n]$. As before, $n+1$ is the sink. Jumping particles move to the sink with probability $q_n \in (0,1]$ and otherwise jump to a uniformly chosen site in $[n]$, including the particle's current location.  
We denote the number of sleeping particles left in a stationary configuration by $S_n = S_{n , q_n, \lambda}$. 

See Appendix~\ref{sec:notation} for explanations of our asymptotic and random variable notation. 
Our first result shows that the critical density $p$ appears robustly across many sink strengths.

\begin{theorem}\thlabel{thm:iff}
 $S_n/n \overset{ \P} \to p$ if and only if $q_n = e^{-o(n)}$.
\end{theorem}

 Since our argument is shorter than the over thirty-page argument in \cite{jarai2024critical}, we prove the analogue of \eqref{eq:1.2} for $S_n$.
Recall that $\alpha_n := \sqrt{p(1-p)n \log n}$.

\begin{theorem}\thlabel{thm:1.2}
    If $q_n = 1/(n + 1)$, then $\P( |S_n - (pn + \alpha_n)| > \epsilon \alpha_n) \to 0$ for all $\epsilon >0$.    
\end{theorem}

Our main result that $S_n$ has a Gumbel scaling limit is sharper and more general than \thref{thm:1.2}.
Throughout this paper we let $\sigma :=\sqrt{ p (1-p)}$ and $b := \log(\sigma/\sqrt{2 \pi})$. Additionally, we set $$ r_n := 1/q_n, \quad a_n := \sigma \sqrt{n},\text{ and }  f_n := \sqrt{ 2 \log(r_n/\sqrt n)}.$$
We scale $S_n$ to 
$$
    G_n := \frac{S_n - pn - a_n f_n}{a_n / f_n}  - b,
$$
and let $G$ denote a random variable with the standard Gumbel distribution $\P(G \le x) = \exp(-e^{-x})$ for all $x \in \R$.

\begin{theorem} \thlabel{thm:gumby}
Assume that
    \begin{align}\label{eq:regime}
n^{1/2 + \epsilon} \ll r_n = \exp(o(n^{1/3})) \text{ for some } \epsilon >0. 
\end{align}
Then, as $n \to \infty$, 
\begin{equation}\label{eq:in-distr}
    G_n \xrightarrow[]{d} G.
\end{equation}
Moreover, 
\begin{equation}\label{eq:var}
    \var(G_n) \to \var(G) = \frac{\pi^2}{6}.
\end{equation}
\end{theorem}

Our last result concerns \emph{fixed-energy} ARW on $\bar K_{n+1}$ in which we set $q_n = 0$ so that no particle ever moves to the sink. Let $\Jumps=\Jumps_{m_n,\lambda}$ be the total number of jumps taken by any particle when stabilizing the configuration with $1\leq m_n\leq n$ active particles placed uniformly at random in $[n]$.

\begin{theorem} \thlabel{thm:density}  Suppose that $m_n = \lceil \mu n\rceil $ for $\mu \in (0,1]$. 
    \begin{enumerate}[label = (\alph*)]
        \item If $\mu >p$, then $\Jumps \geq e^{cn}$ w.h.p.\ for some $c = c(\mu,p)>0$.
        \item If $\mu =p$, then $\Jumps / (n \log n) \overset{\P} \to 1/2$.
        \item If $\mu <p$, then $\Jumps \leq C n$ w.h.p.\ for some
        $C = C(\mu,p)< \infty$.
    \end{enumerate} 
\end{theorem}

Hence $\Jumps$ undergoes a phase transition at $\mu=p$.

\subsection{Discussion}

\thref{thm:gumby} is an exact description of the critical window for a broad range of sink strengths and completes the project initiated by \cite{jarai2024critical}. Moreover, \eqref{eq:var} implies that 
$$
    \var(S_n) \sim \frac{\sigma^2 \, \pi^2}{12} \, \frac{n}{\log(r_n / \sqrt{n})}= \Theta(n/\log r_n)=o(n)
$$
for $r_n$ satisfying \eqref{eq:regime}.
Let $S_n= I_1 + \cdots + I_n$ with $I_v=I_v(n)$ an indicator random variable on site $v$ being occupied.
It follows by rearranging the variance formula for a sum of identically distributed random variables that
\begin{align*}\label{eq:cov}
\cov(I_1,I_2) \sim - \f{\sigma^2}{n}
\end{align*}
is negative for large $n$.

Levine and Silvestri conjectured that ``rigid repulsion among particles'' makes it so that ``particle counts\dots\ have significant negative correlations" for driven-dissipative ARW on boxes in $\mathbb Z^d$ \cite{LS24}. Analogous to the hyperuniform systems---such as {Coulomb systems}, {determinantal processes}, and {Gaussian zero processes}---surveyed by Ghosh and Lebowitz in \cite{ghosh2017fluctuations}, they conjectured that ARW exhibits \emph{hyperuniformity} i.e., sublinear variance (in the graph volume) for the number of sleeping particles at stationarity. 
 \thref{thm:gumby} implies that $S_n$ is hyperuniform. Hence, the law of $(I_1,\hdots, I_n)$ is not a product measure. We are unaware of any other growing family of graphs for which correlations had been proven to occur for ARW. Our conclusions of hyperuniformity along with the exact first order of the covariance are even stronger findings.
 
Another interesting feature of \thref{thm:gumby} is the appearance of the Gumbel distribution. This is a notable departure from the systems in \cite{ghosh2017fluctuations} which all exhibit Gaussian scaling limits. The \emph{least action principle} \cite{rollaSurvey} roughly says that, if the instructions are fixed in advance, ARW minimizes jumps to the sink and thus leaves a maximal amount of sleeping particles in $[n]$ compared to toppling sequences that are permitted to selectively wake sleeping particles. The extremal-type scaling limit in \thref{thm:gumby} gives a quantitative expression of this maximality and raises the question of whether Gumbel scaling limits are a universal feature of ARW.

The \emph{density conjecture} states that the limiting density for driven-dissipative ARW equals the fixed-energy critical value, as established in \thref{thm:iff} and \thref{thm:density}. Originally formulated for sandpile dynamics on boxes in $\mathbb{Z}^d$ versus the phase transition for stabilization time on the torus \cite{dickman1998self}, the conjecture was recently proven for ARW in $d=1$ \cite{hoffman2024density, forien2025newproof} but remains open for $d \geq 2$. This conjecture is a central problem in the foundations of self-organized criticality because it relates the self-organized state with the conventional notion of criticality as a parameterized system held at a phase transition. In \thref{rem:activity}, we explain why \thref{thm:density} implies that the number of jumps at a fixed site tends to infinity at $\mu = p$ while remaining bounded w.h.p.\ when $\mu < p$. This establishes an analogue to the conjecture that ARW on $\mathbb{Z}^d$ stays active when initialized exactly at the critical density \cite[Section 1.3]{rollaSurvey}.

\subsection{Proof Overviews} \label{sec:overview}
We introduce a new approach for analyzing $S_n$.
Similar to \cite{jarai2024critical}, we stabilize $1_{[n]}$ and wait for a fluctuation in the number of sleeping particles to end the stabilization. Rather than attempting to track the intermediate states of the stabilization, we restrict our attention to a count of vertices that each contain a particle trying to fall asleep. This is an elementary Markov chain with the same law as a $p$-biased Ehrenfest urn \cite{karlin1965ehrenfest}.

We call the Markov chain the \emph{binomial update process} $Y(t) = Y_{n,q_n,\lambda}(t)$. We  define it presently without reference to ARW and make the connection explicit in the proof of \thref{prop:stop}.
Initially, $Y(0)$ is distributed as $\mathrm{Binomial}(n,p)$, viewing it as a sum of $n$ independent $\Ber(p)$ coordinates. 
At integer time steps $t\geq 1$, the chain takes a lazy step with probability $q_n$ so that $Y(t) = Y(t-1)$. Otherwise, we choose one of the $n$ coordinates uniformly at random and resample it as an independent $\Ber(p)$ variable, leaving the remaining coordinates unchanged. Set $Y(t) \in \{Y(t) -1, Y(t), Y(t) + 1\}$ to be the total number of coordinates equal to $1$ after this update. Let $Z(t)$ be the total number of lazy steps up to time $t$. Define the boundary $B(t) := n - Z(t)$ and let $$\mathfrak j= \mathfrak j_{n,q_n,\lambda} := \min \{ t \colon Y(t) = B(t)\}$$ be the first hitting time of the boundary. The key to our results is that $S_n$ and $Y(\mathfrak j)$ have the same distribution.

\begin{proposition} \thlabel{prop:stop}
   $S_n \overset{d} = Y(\mathfrak j)$.%
\end{proposition}

The proof of \thref{prop:stop} uses the abelian property of ARW to redirect all jumping particles through an added purgatory vertex where particles do not fall asleep. This makes the count of sites in $[n]$ with sleeping particles evolve with the same law as $Y(t)$ when we hold particles in purgatory as long as possible, with $Z(t)$ corresponding to the number of particles in the sink. Stabilization occurs when all particles are either sleeping on $[n]$ or in the sink, which corresponds to the hitting time $\mathfrak j$.

By \thref{prop:stop}, we can understand $S_n$ in terms of when $Y(t)$ hits the randomly, but predictably decreasing boundary $B(t)$. Since $Y(t)\sim \text{Binomial}(n,p)$ is stationary, the boundary will be hit when $B(t)$ is slightly above $pn$ and $Y(t)$ has an upward fluctuation. This is why we observe extremal-type Gumbel scaling. 
The proofs of Theorems~\ref{thm:iff}--\ref{thm:gumby} make this idea rigorous via progressively sharper interpretations of ``slightly above $pn$'' and ``upward fluctuation.'' The proof of \thref{thm:density} also uses this idea, but in the fixed-energy setting the boundary does not move and the binomial update process starts out of equilibrium.

The proof of \thref{thm:iff} goes by a union bound that uses tail estimates on the binomial distribution.  The proof of \thref{thm:1.2} refines into three stages and uses estimates on the running maximum of $Y(t)$ from \thref{lem:max-tail}. To prove \eqref{eq:in-distr} of \thref{thm:gumby} it is more convenient to work in continuous time. 
Our approach follows Aldous' Poisson clumping heuristic \cite{aldous2013probability}. To make it rigorous we use running maximum and exceedance time estimates for $Y(t)$ and a method in the spirit of the Chen-Stein method for Poisson approximation \cite{arratia1990poisson}. A small technical refinement to the proof of \eqref{eq:in-distr} is used to prove \eqref{eq:var}. The authors of \cite[Section 7]{jarai2024critical} referred to the Ornstein-Uhlenbeck process to heuristically explain the appearance of $\sigma$ in $\alpha_n$.
Though we do not use the connection directly, our calculations are also partly guided by the Ornstein-Uhlenbeck process, which is a scaling limit of $Y(t)$ \cite{karlin1965ehrenfest}. See Section~\ref{sec:setup} for more discussion. 

\subsection{Comparison to \cite{jarai2024critical}}

To prove  \eqref{eq:1.2} the authors of \cite{jarai2024critical} developed and analyzed a Markov chain that starts from $1_{[n]}$ and tracks the total number of particles as well as the total number of active particles. At each step in the chain, an active particle is selected and, if it does not fall asleep, is repeatedly moved until reaching either the sink or an unoccupied site. 
They employed martingale methods to show that the ordered pair of the total and active particle counts stays close to the straight line from $(n,n)$ to $(pn,0)$ and is likely to fixate at a terminal point $(pn+\alpha_n + o(\alpha_n),0)$ when an upward fluctuation in the number of sleeping particles results in a stable configuration. 

An advantage of their approach is that it gives insight into the dynamics of the stabilization as the particle counts descend to the terminal point. However, the complexity of the Markov chain makes analysis difficult. 
Involved calculations are needed to obtain the leading constant $\sigma$ of $\alpha_n$.  Further sharpening their method to obtain lower-order terms appears challenging. As discussed earlier, these terms are important because they characterize the critical state as hyperuniform, distinguishing it from a product measure.

\subsection{Further questions}
Regarding the necessity of hypothesis \eqref{eq:regime}, note that there must be a transition since $q_n=1$ trivially results in $S_n \sim \Bin(n,p)$ which has a Gaussian scaling limit. We believe that our approach could prove that the Gaussian scaling limit persists up until $r_n = O(n^{1/2 - \varepsilon})$ with a transition occuring in the $r_n = n^{1/2 + o(1)}$ window.  As for large values of $r_n$, we believe that the upper bound could be modified to obtain a similar result for $r_n = \exp(\Omega(n^{1/3}))$. We do not study this regime since the connection between the Binomial and Gaussian distributions deteriorates, so formulating the results would become more involved.
We also believe that a similar approach could give a Gumbel scaling limit for the stochastic sandpile model on the complete graph, which is a generalization of \cite[Theorem 2.1]{campailla2026stochastic} that found $1/2$ to be the limiting critical density. In this setting, the sampling method developed in \cite{campailla2026stochastic} changes our binomial update process to the classical Ehrenfest urn model \cite{ehrenfest1907zwei, kac1947random}.

\subsection{Organization}
Section~\ref{sec:BUP} proves \thref{prop:stop} and 
Section~\ref{sec:iff} proves \thref{thm:iff}.
Section~\ref{sec:1.2} sketches the proof of \thref{thm:1.2}.
Section~\ref{sec:density} proves \thref{thm:density}.
Section~\ref{sec:setup} provides the setup and heuristics for the proof of \thref{thm:gumby} including the statement of \thref{lem:max-tail}.
Section~\ref{sec:gumby} proves \thref{thm:gumby}.
Appendix~\ref{sec:m} contains the proof of \thref{lem:max-tail} and Appendix~\ref{sec:notation} explains our asymptotic and random variable notation.%

\section{Proof of \thref{prop:stop}}  \label{sec:BUP}

    Consider the following {purgatory} modification of ARW on $\bar K_{n+1}$ that adds an additional \emph{purgatory} vertex $v$. Begin with the configuration $1_{[n]}$. Throughout, particles on $[n]$ still jump at rate $1$ and fall asleep at rate $\lambda$, but each jump is to $v$. Particles on $v$ do not fall asleep; they jump at rate $1$, moving to the sink $n + 1$ with probability $q_n$ and otherwise to a uniform vertex on $[n]$. Using the site-wise construction of ARW (see for instance \cite{rollaSurvey}, and their notions of the \emph{instruction field}, \emph{stabilizing toppling sequences}, and the \emph{abelian property}) with the distribution for the instruction field as described above, we can see that the final configuration on $[n]$ follows the stationary distribution for the driven-dissipative chain on $\bar K_{n+1}$, without the extra vertex $v$. To see this, using the abelian property, after any particle jumps from $[n]$ to $v$, one may immediately topple $v$, sending the particle to a location on $[n + 1]$ according to the correct distribution. Thus, the number of particles that end up sleeping on $[n]$ has the same law as $S_n$.

    In the following paragraph we define a stabilizing toppling sequence for $1_{[n]}$ on the purgatory graph, breaking the sequence into \emph{steps}. We will let $\bar{Y}(t)$ denote the number of particles on $[n]$ and $\bar{Z}(t)$ the number of particles in the sink after completing step $t$, and define $\bar{B}(t) := n - \bar{Z}(t)$, the number of particles not in the sink. The remaining $n - \bar{Z}(t) - \bar{Y}(t) = \bar{B}(t) - \bar{Y}(t)$ particles are located at $v$. All $\bar{Y}(t)$ particles on $[n]$ after each step $t$ will be sleeping, so the system is stabilized after step $\mathfrak{j}'$, when we first have $\bar{B}(\mathfrak{j}') - \bar{Y}(\mathfrak{j}') = 0$, leaving $\bar{Y}(\mathfrak{j}') = \bar{B}(\mathfrak{j}')$ particles on $[n]$. Thus, it suffices to verify that $\bar{Y}(\mathfrak{j}') \overset{d}{=} Y(\mathfrak{j})$.

    We now describe the procedure, confirming along the way that 
    \begin{equation}\label{eq:purg-coupling}
        \big(\bar{Y}(t), \bar{Z}(t)\big)_{0 \le t \le \mathfrak{j'}} \overset{d}{=} \big(Y(t), Z(t)\big)_{0 \le t \le \mathfrak{j}}.    
    \end{equation}
    (Note that stopping times $\mathfrak{j},\mathfrak{j}'$ are defined identically.) In step $0$, we topple each site of $[n]$ once, so each particle either jumps to $v$ or falls asleep at its starting position. Then indeed $\bar{Y}(0) \sim \Bin(n,p)$, with the sites of $[n]$ acting as $\Ber(p)$ coordinates, and $\bar{Z}(0) = 0$. Now take some $t < \mathfrak{j}'$, so there is at least one particle in purgatory after step $t$. Begin step $t + 1$ by toppling $v$ once. If a particle was sent to the sink, which happens with probability $q_n$, end the step with $\bar{Z}(t+ 1) = \bar{Z}(t-1) + 1$ and $\bar{Y}(t + 1) = \bar{Y}(t)$, a lazy step for $Y(t)$. Otherwise, a particle is sent to a uniformly sampled site $i_t \in [n]$, which now has one or two active particles. If there are two, topple $i_t$ until a particle jumps to $v$. Either way, once one active particle remains, end the step by toppling $i_t$ once, putting the particle to sleep or sending it to $v$. Then $\bar{Y}(t + 1)$ differs from $\bar{Y}(t)$ in that coordinate $i_t$ is resampled, and $\bar{Z}(t + 1) = \bar{Z}(t)$. Thus, \eqref{eq:purg-coupling} holds, so indeed $Y(\mathfrak{j}) \overset{d}{=} \bar{Y}(\mathfrak{j}')$, which as discussed has the same law as $S_n$. \qed

\section{Proof of \thref{thm:iff}} \label{sec:iff}

Call time $t \ge 0$ a \emph{boundary exceedance} time if $Y(t) \ge B(t) = n - Z(t)$ (we let $\big(Y(t), B(t)
\big)$ run for all time $t \ge 0$, not just up until $\mathfrak{j}$). We note that $\mathfrak{j}$ may be equivalently defined as the first boundary exceedance time instead of the first boundary hitting time, since $Y(t) - B(t)$ changes by $0$ or $1$ at each time step, and at the start, $Y(0) \le n = B(0)$.  Recall $r_n = 1/q_n$ and suppose that $r_n = e^{o(n)}$. For $b \in (-\infty,  n]$, let $t(b)$ denote the first time $t$ that $B(t) \le b$, which is distributed as the sum of $\lceil n - b \rceil$ independent $\Geo(q_n)$-distributed random variables (and is decreasing in $b$). 

Necessarily $\mathfrak{j} \le t(0)$, since $Y(t) \ge 0$ at all $t \ge 0$. And w.h.p.\ $t(0) < \lceil 2 n r_n \rceil =: t^+$, since $t(0)/n \xrightarrow{\P} r_n$ by the law of large numbers. This implies that $Y(\mathfrak j)$ w.h.p.\ lies between the running minimum and maximum of $Y(t)$ after $t^+$ steps. As $Y(t) \sim \text{Binomial}(n,p)$ for all $t \ge 0$, by a union bound and using \thref{prop:stop}, we have for all $\varepsilon > 0$ that 
\begin{align}\label{eq:union-bound-subexp}
    \P\big(\big|S_n/n - p\big| \ge \varepsilon\big) 
        &= \P\big(\big|Y(\mathfrak{j})/n - p\big| \ge \varepsilon\big) \nonumber \\
        &\le \P(\mathfrak{j} \ge t^+) + t^+ \, \P\big(\big|Y(0) / n - p\big| \ge \varepsilon\big) \nonumber \\
        &= o(1) + t^+ \, \P\big(\big|Y(0)  - pn\big| \ge \varepsilon n \big).
\end{align}
To show that \eqref{eq:union-bound-subexp} vanishes, we use Cram\'er's theorem, viewing $Y(0)$ as a sum of independent $\Ber(p)$ random variables \cite{cramer2018new}. For $a \in (0,1)$ let
$$
    D(a \, \| \, p) := a \, \ln \frac{a}{p}+(1-a) \ln \frac{1-a}{1-p},
$$
which is positive away from $p$. Then we have, for $0 < \varepsilon <  \min(p, 1 - p)$, 
\begin{equation}\label{eq:cramer-pos}
    \P\big(Y(0)  \ge (p + \varepsilon)n\big) = e^{-n\big(D(p + \varepsilon \, \| \, p) + o(1)\big)}  = e^{-\Omega(n)}
\end{equation}
and 
\begin{equation}\label{eq:cramer-neg}
    \P\big(Y(0) \le (p - \varepsilon)n\big) = e^{-n\big(D(p - \varepsilon \, \| \, p) + o(1)\big)}  = e^{-\Omega(n)}.
\end{equation}
Indeed then, \eqref{eq:union-bound-subexp} vanishes since $t^+ = e^{o(n)}$. Therefore, $S_n /n \xrightarrow{\P} p$.

Now suppose that there exists $c>0$ such that $r_n \gg e^{cn}$. By \eqref{eq:cramer-pos}, we may choose $\varepsilon > 0$ small enough so that $\P\big(Y(0) \ge (p + \varepsilon)n\big) \ge e^{-cn/4}$ for all large $n$. We will show that 
\begin{equation}\label{eq:non-meanfield}
    Y(\mathfrak{j}) \ge \lfloor (p + \varepsilon)n \rfloor \quad \text{w.h.p.,}
\end{equation}
which establishes the second half of the theorem.

 Let $\tau_0 := t((p + \varepsilon)n)$, so $B(\tau_0) = \lfloor (p + \varepsilon)n \rfloor$. We will show w.h.p.\ that a boundary exceedance occurs before any additional visits the sink. For $k \ge 1$, let $\tau_{k}$ be the first time after $\tau_{k-1}$ that all $n$ of the coordinates of $Y(t)$ are resampled. It follows from the definition of $Y(t)$ that the $Y(\tau_k)$ are independent. Let $K$ be the smallest index $k$ such that $Y(\tau_k) \geq (p+\varepsilon)n$, which is a boundary exceedance as $B(t)$ is non-increasing. Thus, $\mathfrak{j} \le \tau_K$, and so 
 \begin{equation}\label{eq:B-tau-k}
     Y(\mathfrak{j}) = B(\mathfrak{j}) \ge B(\tau_K).
 \end{equation}
 By our choice of $\varepsilon$ and independence of the $Y(\tau_k)$, $K$ is dominated by a $\Geo(e^{-cn/4})$-distributed random variable. Hence, $K \leq e^{cn/3}$ w.h.p.\ which implies $\tau_K \le \tau_{\lfloor e^{cn/3} \rfloor}$.
Moreover, since $\tau_{k} - \tau_{k-1}$ is the classical coupon collection time on $n$ coupons (plus some rare additional lazy steps where no coupons are collected), and $\tau_{\lfloor e^{cn/3} \rfloor} - \tau_0$ is the sum of $\lfloor e^{cn/3} \rfloor$ independent collection times, we have $\tau_{\lfloor e^{cn/3} \rfloor} - \tau_0$ concentrates around $e^{cn/3}(n \log n)/(1 - q_n) \sim e^{cn/3}(n \log n)$ and hence is no larger than $e^{cn/2}$ w.h.p. 
For large $n$, it takes at least a $\Geo(e^{-cn})$-distributed number of jumps between each particle accumulation at the sink. Thus, w.h.p.\ no additional particles have moved to the sink on the time interval $[\tau_0, \tau_0 + e^{cn/2}]$, nor then on the interval $[\tau_0, \tau_K]$ w.h.p., in which case $B(\tau_K) = \lfloor (p + \varepsilon)n \rfloor$. Along with \eqref{eq:B-tau-k} this establishes \eqref{eq:non-meanfield}, finishing the proof.
\qed 

\begin{remark}\thlabel{rem:exponential}
    If $r_n = e^{(a + o(1))n}$ for some $0 < a <  \ln (1/p)$, by refining the above argument one could show that $S_n / n \xrightarrow{\P} p + \beta$ where $\beta > 0$ is the unique value so that $D(p + \beta \, \| \, p) = a$. And if $a > \ln (1 / p)$ one would get $S_n = n$ w.h.p.
\end{remark}

\section{Proof sketch of \thref{thm:1.2}} \label{sec:1.2}
  Suppose that $r_n = n + 1 =1/q_n $. 
  As above, let $t(z)$ be the first time step that $B(t) \le z$ for $z \leq n$. Define $M_{z,y} := \max_{t(z) \leq t < t(y)} Y(t)$. Recalling that $\alpha_n = \sqrt{p(1-p)n  \log n },$ fix $0< \epsilon < 1$ and define the decreasing thresholds
\begin{alignat*}{2}
    k_1 &:= pn + 2\alpha_n ,          &\qquad k_2 &:= pn + (1+\epsilon)\alpha_n , \\
    k_3 &:= pn - (1-\textstyle{\frac\epsilon2})\alpha_n, &\qquad k_4 &:= pn - (1-\epsilon)\alpha_n .
\end{alignat*}
It suffices to prove w.h.p.\ that
\begin{align}
    M_{n, k_1} &< k_1 \label{eq:jarai-goal1}, \\
    M_{k_1, k_2} &< k_2 \label{eq:jarai-goal2}, \quad \text{and} \\
    M_{k_3, k_4} &\ge k_3 \label{eq:jarai-goal3}.
\end{align}
\eqref{eq:jarai-goal1} and \eqref{eq:jarai-goal2} ensure that $\mathfrak{j} \ge t(k_2)$ w.h.p.\ while \eqref{eq:jarai-goal3} ensures that $\mathfrak{j} < t(k_4)$ w.h.p. Therefore, $k_2 \ge Y(\mathfrak{j}) > k_4$ w.h.p., which proves the theorem. 

All three statements follow straightforwardly from \thref{lem:max-tail} (the lemma is stated for a continuous-time version of $\big(Y(t), Z(t)\big)$, which is no more complicated than the discrete-time version). \thref{lem:max-tail} is stated for deterministic, not random, times, but as in the proof of \eqref{eq:thm-main-goal}, we may exploit concentration estimates on $t(k_1)$, $t(k_2) - t(k_1)$, and $t(k_4) - t(k_3)$, which in this case are independent sums of many exponential random variables \cite{janson2018tail}.
        \qed 

\begin{remark}\thlabel{rem:jarai}
    ARW on $\bar{K}_{n+1}$ with sink-strength $1/(n+1)$, so that particles perform symmetric walks, corresponds to the ARW setting from \cite{jarai2024critical} on $K_{n+1}$ with a slightly increased sleep probability. Specifically, discarding the rate-$(1/(n+1))$ self-loops, which are lazy steps regarding the configuration, the sleep probability on $K_{n+1}$ is given by $\frac{\lambda}{\lambda + 1 - 1/(n+1)} = p + O(1/n)$. The $O(1/n)$ error is absorbed into $\varepsilon \alpha_n$, thus implying \eqref{eq:1.2}. However, some minor care would be needed to prove \thref{thm:1.2} for non-fixed $p$ (bounded away from $0$ and $1$). Namely, \thref{lem:max-tail} would need to be stated and proved in this more general setting, which should add no difficulty outside of additional bookkeeping. 
\end{remark}

\section{Proof of \thref{thm:density}} \label{sec:density}
Start with a configuration with $m_n$ particles added uniformly to $n$. Topple each site of $[n]$ until all particles are asleep or in purgatory. 
Let $Y(0) \in [0, m_n]$ denote the number of particles asleep on $[n]$. A straightforward adaptation of \thref{prop:stop} gives that we may run the binomial update process dynamics starting from $Y(0)$ with $q_n = 0$, and $\Jumps$ is the first time such that $Y(t) = m_n$. This is because each binomial update step corresponds to a single jump from purgatory, and the jumps from purgatory are in bijection with the jumps in the original process.

Notice that $Y(t)$ is a skip-free lazy random walk with \emph{drift}
\begin{align} \label{eq:Y}
\E[Y(t+1) - Y(t)\mid Y(t) = y] &= (-1) \f{y}{n}(1-p)  + (1) \f{n-y}{n} p = p - \f y n.
\end{align}
The probability of a lazy step is $$\frac{y}{n}p + \frac{n-y}{n}(1-p) \leq p_* := \max(p,1-p).$$

Suppose that $\mu = p-\epsilon$ for some $\epsilon >0$ so that $m_n = \lceil(p - \epsilon)n\rceil$. At worst case, the walk begins at $Y(0) = 0$. When $y \leq \lceil (p-\epsilon)n \rceil$, by \eqref{eq:Y} the drift is at least $\epsilon$. Thus, there is a coupling where $\Jumps$ is at most the first passage time from $0$ to $m_n$ of a $p_*$-lazy walk where non-lazy steps have drift $\epsilon$, which concentrates on $(1 - p_*)\epsilon^{-1} pn$. Thus, $\Jumps \le 2(1 - p_*)\epsilon^{-1}pn$ w.h.p.

Now suppose that $\mu = p+\epsilon$ for some $\epsilon >0$ so that $m_n = \lceil(p+\epsilon)n\rceil$. We claim that $Y(t)$ is stochastically dominated by the $\Bin(n,p)$ distribution for all $t \ge 0$. Indeed, if we let $n_t \in [0,n]$ denote the number of sites visited at any point by time $t$, then 
\begin{equation}\label{eq:cond-on-coupons}
    Y(t) \mid n_t \sim \Bin(n_t,p).
\end{equation}
By Hoeffding's inequality then, $\P(Y(t) \ge m_n) \le e^{-2\varepsilon^2 n}$ for any $t \ge 0$. A union bound yields
\begin{align*}
    \P(\Jumps < e^{\varepsilon ^2 n }) = \P\bigg(\bigcup_{0}^{e^{\varepsilon ^2 n } - 1}\{Y(t) \ge m_n\}\bigg) \le e^{-\epsilon^2 n}.
\end{align*}

Finally, suppose that $\mu = p$ so that $m_n = \lceil pn\rceil $. Note that $n_t$, the number of visited sites at time $t$, begins with $n_0 \in [0, m_n]$ and from there increases as the coupon collector problem on $n$ coupons. Fix $0<\delta<1/6$ and let $t' := \lceil (\f 12 - 3\delta)n \log n\rceil $. 
Let $E'$ be the event that $n_{t'} \le n - n^{\f 12 + 2 \delta}$, which occurs w.h.p.\ using that $n_0 \le \lceil pn\rceil$ begins with a positive fraction of unvisited sites. On $E'$, $n_{t} \le n - n^{\f 12 + 2 \delta}$ holds for all $0 \le t \le t'$. By \eqref{eq:cond-on-coupons} and moderate deviation estimates 
\cite[Theorem 1]{cramer2018new} for the binomial distribution then, we have for $0 \le t \le t'$ that 
$$
    \P(Y(t) \ge m_n \mid E') = O(e^{-n^{4 \delta}/2}).
$$
By a union bound then, we have 
$$
    \P(\Jumps < t' \mid E') \le t'O(e^{-n^{4 \delta}/2}) = o(1).
$$
Thus $\Jumps \ge t'$ w.h.p.\ recalling that $\P(E') \to 1$.

Let $t'' := \lceil (\f 12 + 2\delta)n \log n\rceil $, 
and $t_0 =\lceil (\f 12 + \delta)n \log n\rceil$.
Then $n_{t_0} \ge n - \lceil n^{1/2}\rceil$ w.h.p., in which case 
$Y(t_0)$ dominates a Binomial$(n-\lceil n^{1/2} \rceil, p)$ random variable. Chebyshev's inequality implies that $Y(t_0) \geq pn - \sqrt n \log \log n$ w.h.p.\ The number of steps, after $t_0$, that it takes  for $Y(t)$ to first reach $m_n$ is dominated by the time it takes for a $p_*$-lazy simple random walk started at $0$ to exceed $\sqrt n \log \log n$. This is because, by \eqref{eq:Y}, there is always a slight upward drift when $Y(t) < pn$. The reflection principle for the running maximum of a simple random walk implies that w.h.p.\ during the next $t'' - t_0 \sim \delta n \log n$ steps the walk will at some time exceed $\sqrt n \log \log n$. Thus, $\Jumps \leq t''$ w.h.p.
\qed

\begin{remark}\thlabel{rem:activity}
    Let $\Jumps(x)$ be the total number of jumps that occur at site $x \in [n]$. We claim that $\Jumps(x) \geq \f 14 \log n$ w.h.p.\ when $\mu = p$. 
    Letting $U_n(x, t)$ denote the number of updates to coordinate $x$ by time $t \ge 0$ in the binomial update process, we have $U_n(x, t) \sim \Bin(t, 1/n)$. And provided $t \le \Jumps$, $U(x, t)$ is at most one more than $\Jumps(x)$ in our coupling, as it counts the number of jumps to $x$ that have occurred. The statement follows as $\Jumps \ge \tfrac13 n \log n$ w.h.p.\ by \thref{thm:density}. 
    
    Similar reasoning implies in the $\mu < p$ setting that $\Jumps(x)$ is bounded w.h.p.\ by any sequence $\omega_n \to \infty$.
\end{remark}

\section{Setup for the proof of \thref{thm:gumby}} \label{sec:setup}

In this section we introduce additional notation that we use to restate claim \eqref{eq:in-distr} from  \thref{thm:gumby}. Next, we introduce a continuous-time binomial update process and state the analogue of \thref{prop:stop}. We then discuss the running maximum of the scaled update process and its connection to the Ornstein-Uhlenbeck process, as well as state a lemma regarding the running maximum. Lastly, we provide an informal approximation that explains the appearance of a Gumbel scaling limit which will be made rigorous in the next section.  

\subsection{Restatement}
Recall that $a_n = \sigma\sqrt{n}$ where $$\text{$\sigma = \sqrt{p(1-p)}$, \quad  $f_n = \sqrt{2\log\big(r_n / \sqrt{n}\big)}$, \quad $G_n = \frac{S_n - pn - a_n f_n}{a_n / f_n}  - b$},$$ and $b = \log(\sigma / \sqrt{2\pi})$. Define the normalization
$$
    s_n := \frac{S_n - pn}{a_n} = f_n + (G_n + b)/f_n.
$$
For $x_0 \in \R$, we have $G_n \le x_0$ if and only if $s_n \le y_0$ for 
\begin{equation}\label{eq:y0}
    y_0 := f_n + (x_0 + b)/f_n.
\end{equation}
Therefore, convergence in distribution is equivalent to the statement that 
	\begin{equation}\label{eq:thm-main-goal}
		\P(s_n \le y_0) \to e^{-e^{-x_0}} \quad \text{for all } x_0 \in \R.
	\end{equation}

\subsection{Continuous-time binomial update process}

We will use a \emph{continuous-time binomial update process} $S(t)$ for $t \in [0,\infty)$ to prove \eqref{eq:thm-main-goal}. In the continuous-time process, each coordinate independently updates according to a rate-$1$ Poisson process and the sink particle count $X(t)$ updates independently at rate
$$q'_n := \frac{nq_n}{1 - q_n} \sim nq_n,$$
with the asymptotic equivalence with $nq_n$ coming from assumption \eqref{eq:regime}. This makes it so $X(t)$ updates with probability $q_n'/(n + q'_n) = q_n$.
For $t \ge 0$, we let $\tau(t) \approx t + \log n$ be the next \emph{regeneration time}, i.e., the first time such that each coordinate has been updated at some time on $(t, \tau(t)]$, as in the coupon collector problem.
Note that $s(t) = (S(t) - pn)/a_n$ is a stationary, reversible process with $s(0) \xrightarrow{d} \mathcal{N} \sim \mathcal{N}(0,1)$ as $n \to \infty$. In fact, $s(t)$ scales to a standard Ornstein-Uhlenbeck (OU) process, i.e.\ the process $x_t$ satisfying $dx_t = -x_t\, dt + \sqrt{2}\, dW_t$ \cite{karlin1965ehrenfest}.

Denote the boundary by $B'(t) := n - X(t)$, and let $\mathfrak t$ be minimal so that $S(\mathfrak t) = B'(\mathfrak t)$. We restate \thref{prop:stop} as
$$S_n \overset{d}{=} S(\mathfrak t) = B'(\mathfrak t).$$ 
Denote the rescaled boundary by 
$$b(t) := \f{B'(t) - pn}{a_n}.$$ 
Note then that $\mathfrak t$ may be defined as the first time that $s(t) = b(t)$ and equivalently that $s(t) \ge b(t)$. As with the discrete-time process, we call any time $t$ where $s(t) \ge b(t)$ a boundary exceedance time. 

\subsection{The running maximum} 
For $t' \leq t$ we define $$m(t',t) := \max_{t' \leq r \leq t} s(r)$$ to be the running maximum on the interval $[t',t]$ with the convention that $m(t) = m(0,t)$. 
We will break time into short enough intervals where $b(t)$ is essentially flat, so that the probability of a boundary exceedance during a length-$T$ interval is well-approximated by the probability $m(T)$ exceeds some constant.

In the OU process, exceedances above a height $x \gg 0$ occur like Poisson point processes, where once the height $x$ is reached, it will spend an average of $\sim 1 / x^2$ time above before mixing occurs, after which there is another long exponential waiting time before the next exceedance \cite{berman1982sojourns}. By stationarity, the proportion of time above $x$ is given by 
$$
    \Tail(x) := \P(\mathcal{N} \ge x) \sim \frac{\exp(-x^2/2)}{\sqrt{2 \pi}x},
$$
with the asymptotic notation denoting that $x \to \infty$. Exceedances then come in at approximately rate $x^2 \bar{F}(x)$. So we let
$$
    \mu_x := \frac{x\exp(-x^2/2)}{\sqrt{2 \pi}} \sim x^2\Tail(x),
$$
and record here that $\mu_x$ is decreasing for $x \in [1, \infty)$. A corresponding story holds for $s(t)$ if $x \ll n^{1/6}$, where the tails of $s(t)$ and $\mathcal{N}$ are close in ratio. Specifically, we prove in Appendix~\ref{sec:m} the following tail bound for the running maximum.

\begin{lemma}\thlabel{lem:max-tail}
    For $x=x_n$ and $T=T_n$ satisfying  $x = \Omega(\sqrt{\log n})$, $x = o(n^{1/6})$, and $\log n \ll T \ll 1/\mu_x$, 
    \begin{equation*}
        \P(m(T) \ge x) \sim \mu_x \,T.
    \end{equation*}
\end{lemma}
Given the connection between the binomial update process and the Ehrenfest urn model, this seems like it should be a known result, but we did not find a reference.

\subsection{Heuristic sketch of \eqref{eq:thm-main-goal}}

We perform a back-of-the-envelope justification of \eqref{eq:thm-main-goal} before carrying out its rigorous proof in the next section. For $y \in \R$, let 
$$t_y:= \inf\{ t \ge 0 \colon b(t) \le y\}$$
be the hitting time of level $y$ for the scaled boundary.
Then $s_n \le y$ if and only if there are no boundary exceedances before time $t_y$. Moreover, as we argue in the proof, by \thref{lem:max-tail} it is unlikely to have a boundary exceedance before time $t_{\bar{y}}$ for $\bar{y} := \sqrt{10 \log (n r_n)}$.

Take $y \gg 1$, so that (clustered) boundary exceedances before time $t_y$ behave like a Poisson point process with increasing intensity. Then if $y \le \bar{y}$ we should get that 
\begin{equation}\label{eq:approx-prob}
    \P(s_n \le y) \approx e^{-\lambda_y}
\end{equation}
where $\lambda_y  = \int_{t_{\bar{y}}}^{t_y}\mu_{b(t)} \, dt$. Approximating $b(t)$ by a straight line with slope $-q'_n/a_n$ then, 
\begin{align}\label{eq:approx-lambd}
    \lambda_y  \approx \int_{0}^{t_y - t_{\bar{y}}}\mu_{y +  q'_n / a_n t} \, dt \approx \frac{a_n}{\sqrt{2\pi} \, q'_n}(e^{-y^2/2} - e^{-\bar{y}^2/2}) \approx \frac{\sigma r_ne^{-y^2/2}}{\sqrt{2\pi n}}.
\end{align}

Combining Equations \eqref{eq:approx-prob} and \eqref{eq:approx-lambd}, and plugging in $y_0$ from \eqref{eq:y0} we see that $\P(s_n \leq y_0) \approx e^{-e^{-x_0}}$.

\section{Proof of \thref{thm:gumby}} \label{sec:gumby}

Our main task is proving \eqref{eq:in-distr}, convergence in distribution, which is equivalent to the restatement \eqref{eq:thm-main-goal}, that $\P(s_n \le y_0) \to e^{-e^{-x_0}}$ for all $x_0 \in \R$. We do so in Sections~\ref{sec:Poissonization}--\ref{sec:LB}. In Section~\ref{sec:variance} we make technical refinements to deduce the claimed convergence of $\var(G_n)$ in \eqref{eq:var}. 

For convenience, we recall a few key terms from Section~\ref{sec:setup}.
We have $s_n = (S_n -pn)/a_n$ with $a_n = \sigma\sqrt{n}$ and $\sigma^2 = p(1-p)$. The process $s(t) = (S(t) - pn) / a_n$ is the normalized continuous-time binomial update process and $b(t) = (B(t) - pn)/a_n$ is the rescaled continuous-time boundary process. And the maximum over an interval is given by $m(t',t) = \max_{t' \leq r \leq t} s(r)$ with $m(t) = m(0,t)$. Given $y \in \mathbb R$, the stopping time $t_y := \inf\{t \colon b(t) \leq y\}$ is the first time $b(t)$ reaches level $y$.  Some important quantities and their asymptotic rates are:
	\begin{align}\label{eq:fn}
		f_n &= \sqrt{2\log\big(r_n / \sqrt{n}\big)}= o(n^{1/6}),\\
        f_n &= \Omega(\sqrt{\log n}), \label{eq:fn-lb}\\
		\mu_x &= \frac{x\exp(-x^2/2)}{\sqrt{2\pi}}, \nonumber \\
		q'_n &= \frac{nq_n}{1 + q_n} \sim nq_n. 
	\end{align}
    We take $x_0 \in \R$, but with an eye towards controlling the variance let it depend on $n$, provided it stays in the following range:
    \begin{equation}\label{eq:xn-range}
        -2\log \log n^3 \le x_0 \le 3 \log n.
    \end{equation}
    Then, 
    \begin{equation}\label{eq:y0}
       y_0 = f_n + (x_0 + b)/f_n = \Theta(f_n) 
	\end{equation}
    holds uniformly.

\subsection{Poissonization}\label{sec:Poissonization}

	Let  $$y_k := y_0 + \frac{k}{f_n \log^3 \! n}$$
	for $k \ge 0$, and choose $N$ minimal so that $y_N \ge \max(\bar{y}, 2y_0)$ for $\bar{y} := \sqrt{10 \log (n r_n)}$. Then
    \begin{equation}\label{eq:yN}
        y_N = \Theta(f_n),
    \end{equation}
    and, as $y_N \ge \bar{y}$, 
    \begin{equation}\label{eq:mu-yN}
		\mu_{y_N} = o\bigg(\frac{\sqrt{\log (nr_n)}}{(nr_n)^5}\bigg).
	\end{equation}
    We assume $n$ is large enough that $y_0 \ge 1$, and thus $\mu_{y_k}$ is decreasing in $k$. 
    
    Let $t_k = t_{y_k}$ for $0 \le k \le N$, which are are decreasing times in $k$. Set $t_{N+1} = 0$. Let $J_k := [t_{k+1}, t_k]$ and $T_k := t_{k+1} - t_k$ for $0 \le k \le N$. We scale by $f_n \log^3 \! n$ to define the $y_k$ in order to make $T_k$ much larger than the regeneration time but small enough that $\mu_{b(t)}$ remains flat over $J_k$.

	For $0 \le k \le N$, let $X_k$ be the indicator variable that there is a boundary exceedance on interval $J_k$, with $p_k := \E X_k$. Note that $s_n \le y_0$ if and only if $X := \sum_{k=0}^N X_k = 0$. Let $\nu := \E X = \sum_{k=0}^N p_k$.
	
	The variables $X_k$ are not independent. However, we approximate them from below by indicators $X'_k \le X_k$ that express that on $J_k$, first every coordinate is updated, and then there is a boundary exceedance. In other words, $X'_k$ indicates a boundary exceedance on $[\tau(t_{k+1}), t_k]$, which presupposes that $\tau(t_{k+1}) \le t_k$. To see the $X'_k$ are independent, it is easiest to check $X'_{k + 1}$ is independent of $(X'_{k}, \ldots X'_{0})$. This is because $(X'_{k}, \ldots X'_{0})$ is measurable with respect to 
    $\big(s(t), b(t)\big)_{t \ge \tau(t_{k + 1})}$,
    which is independent of $\big(s(t), b(t)\big)_{t_{k+2} \le t \le t_{k+1}}$, which $X'_{k+1}$ is measurable with respect to. We define the analogous quantities $p'_k := \E[X'_k]$, $X' = \sum_{k = 0}^N X'_k$, and $\nu' = \E[X']$. 

	We will show that 
	\begin{equation}\label{eq:lambd}
		e^{-x_0} \lesssim \nu' \le \nu \lesssim e^{-x_0},
	\end{equation}
    uniformly over \eqref{eq:xn-range}. Note that the minimum of $e^{-x_0}$ on this range is $1/n^3$, so any smaller error may be absorbed. We also show
	\begin{equation}\label{eq:pk-small}
		\max_{0 \le k \le N}p_k \to 0
	\end{equation}
    uniformly, which implies then the same fact about $p'_k \le p_k$.

    We explain now why \eqref{eq:lambd} and \eqref{eq:pk-small} imply \eqref{eq:thm-main-goal} and therefore \eqref{eq:in-distr}. We consider fixed $x_0 \in \R$, and so $\nu',\nu \to e^{-x_0}$ by \eqref{eq:lambd}. Then, since $X - X'$ is a non-negative integer, by Markov's inequality we have 
    \begin{align*}\label{eq:X-X'-close}
        \P(X \neq X') = \P(X - X' \ge 1) \le \E[X - X'] = \nu - \nu' \to 0.
    \end{align*}
    And since the $X'_k$ are independent, \eqref{eq:lambd} and \eqref{eq:pk-small} give that $X' \xrightarrow{d} \Pois(e^{-x_0})$, and thus $X \xrightarrow{d} \Pois(e^{-x_0})$. Therefore, 
	$$
		\P(s_n \le y_0) = \P(X = 0) \to \Pois(e^{-x_0})(\{0\}) = e^{-e^{-x_0}}.
	$$
	Thus, to establish \eqref{eq:in-distr} it remains to show \eqref{eq:lambd} and \eqref{eq:pk-small}.

\subsection{Upper bounds}
    
	We obtain upper bounds on $p_k$ and $\nu$ using \thref{lem:max-tail} and some approximations. Let $T_k$ be the length of interval $J_k$, which is distributed as the sum of independent $\Exp(q'_n)$ random variables, precisely $a_n / (f_n \log^3 \! n)$ of them rounded to an integer, which diverges since $a_n/(f_n \log^3 \! n) = \Omega(n^{1/4})$ by assumption \eqref{eq:regime}. (This also gives that $N = o(n)$). Thus, $T_k$ concentrates at 
	\begin{equation}\label{eq:T}
		T := \frac{1}{q'_n} \cdot \frac{a_n}{f_n \log^3\!n} \sim \frac{\sigma r_n}{\sqrt{n}f_n \log^3 \! n}.
	\end{equation}
	Specifically, if we take $T_\delta := T / \log n$, $T^+ := T + T_\delta$, and $T^- := T - T_\delta$, it is easily shown that
	\begin{equation}\label{eq:Ek}
		\P(T_k \notin (T^-, T^+)) = o(1/n^4)
	\end{equation}
 by standard concentration estimates for a sum of independent $\Exp(q'_n)$ random variables \cite{janson2018tail}. We note here that
	\begin{equation}\label{eq:coupon}
		T_\delta \gtrsim n^{\varepsilon / 2},
	\end{equation}
	a much larger scale than the $O(\log n)$ regeneration time.

	We now establish that
	\begin{equation}\label{eq:pk-ub}
		p_k \lesssim T\mu_{y_k} + o(1/n^4)
	\end{equation}
	for $0 \le k < N$. Note that if $X_k = 1$ then either $T_k \ge T^+$, accounting for the $o(1/n^4)$ term, or $m_k := m(t_{k+1}, t_{k+1} + T^+) \ge y_k$. And by stationarity, $m_k \overset{d}{=} m(T^+)$, so it remains to show 
    \begin{equation}\label{eq:exceedance-fixed-time}
        \P(m(T^+) \ge y_k) \sim T\mu_{y_k}.
    \end{equation}
    This follows from \thref{lem:max-tail} once the conditions of the lemma are verified. 
    
    The conditions on $y_k$ are met as $y_0 \le y_k \le y_N$, using \eqref{eq:fn}--\eqref{eq:fn-lb} and \eqref{eq:y0}--\eqref{eq:yN}. We also have $T^+ \gg T_\delta \gg \log n$ by \eqref{eq:coupon}. Finally, to see $T^+ \ll 1 / \mu_{y_k}$, we have $T^+ \mu_{y_k} \le T^+ \mu_{y_0} \sim T \mu_{y_0}$, and,
	\begin{align}\label{eq:T-mu-y0}
        T\mu_{y_0} = 
		T \, \frac{y_0}{\sqrt{2\pi}}\exp(-y_0^2/2) 
            &= T \, \frac{y_0}{\sqrt{2\pi}}\exp(-f_n^2/2 - (x_0 + \log(\sigma/\sqrt{2\pi})) + o(1)) \nonumber \\
			&\sim T \, \frac{y_0}{\sqrt{2\pi}} \frac{\sqrt{n}}{r_n}e^{-x_0} \frac{\sqrt{2\pi}}{\sigma}\nonumber \\
			&\sim \frac{y_0 \, e^{-x_0}}{f_n \log^3 
            \! n},
	\end{align}
    using \eqref{eq:T} for the last line. This is indeed $o(1)$ as $y_0 = \Theta(f_n)$ and $e^{-x_0} = O(\log^2 n)$, establishing \eqref{eq:exceedance-fixed-time} and thus \eqref{eq:pk-ub}. Moreover, $p_k \to 0$ uniformly over $0 \le k < N$ again by \eqref{eq:pk-ub} and \eqref{eq:T-mu-y0}. Note also that $T \mu_{y_0} = o(e^{-x_0})$ by \eqref{eq:T-mu-y0}, so combined with \eqref{eq:pk-ub} we have
    \begin{equation}\label{eq:mu-y0-e-x0}
        p_0 = o(e^{-x_0}).
    \end{equation}
    
    We also show 
	\begin{equation}\label{eq:pN}
		p_N = o(1/n^{4}) = o(e^{-x_0}).
	\end{equation}
	For $X_N$ to equal $1$, we must have that $m(T_N) \ge y_N$. The time $T_N = t_n$ is larger than the other $T_k$, but it is still bounded by the time it takes all $n$ particles to go to the sink, which concentrates around $n / q'_n \sim r_n$, and so $T_N \le 2r_n$ with at least $1 - o(1/n^{4})$ probability. It suffices then to show $\P(m(2r_n) \ge y_N) = o(1/n^4)$. By \thref{lem:max-tail}, it is enough that $2r_n \mu_{y_N} = o(1/n^4)$, which is the case by \eqref{eq:mu-yN}. Thus \eqref{eq:pk-small} is shown, and only \eqref{eq:lambd} remains.

	We now upper bound $\nu$. By \eqref{eq:pk-ub}, \eqref{eq:mu-y0-e-x0}, \eqref{eq:pN}, along with the fact that $N = o(n)$ and $e^{-x_0} = \Omega(1/n^3)$, we have
	\begin{equation}\label{eq:lambd-sum}
		\nu = \sum_{k=0}^N	p_k \lesssim T\sum_{k=1}^{N-1}\mu_{y_k} + o(1/n^4) + o(N/n^4) = T\sum_{k=1}^{N-1}\mu_{y_k} + o(e^{-x_0}).
	\end{equation}
	The summation term is a Riemann sum for $\mu_y$, which is decreasing on the range of interest. Therefore,
	\begin{equation}\label{eq:riemann}
		T\sum_{k=1}^{N-1}\mu_{y_k} \le Tf_n \log^3 \! n \int_{y_0}^\infty \mu_y \, dy = Tf_n \log^3 \! n \,\frac{\mu_{y_0}}{y_0} \sim e^{-x_0},
	\end{equation}
	where at the end we used \eqref{eq:T-mu-y0}. Equations \eqref{eq:lambd-sum} and \eqref{eq:riemann} establish the upper bound of \eqref{eq:lambd}. \qed
	
\subsection{Lower bound}\label{sec:LB}
	
	To complete the proof of \eqref{eq:in-distr}, we show the lower bound on $\nu' = \sum_{k = 0}^Np'_k \ge \sum_{k = 0}^{N - 1}p'_k$ from \eqref{eq:lambd}. For fixed $0 \le k < N$, we have $X'_k = 1$ as long as the following hold:
	\begin{enumerate}
		\item $m'_k := m(t_{k+1} + T_\delta, t_{k+1} + T^-) \ge y_{k+1}$,
		\item $\tau(t_{k+1}) \le t_{k+1} + T_\delta$,
		\item $t_k > t_{k+1} + T^-$.
	\end{enumerate}
	Denote these events by $E_1, E_2,$ and $E_3$, respectively. Then
    
	\begin{equation}\label{eq:p'k}
		p'_k \ge \P(E_1) - \P(E_2^c) - \P(E_3^c).
	\end{equation}
	By \thref{lem:max-tail}, 
    
	\begin{equation}\label{eq:E1}
		\P(E_1) \sim (T^- - T_\delta)\mu_{y_{k+1}} \sim T \mu_{y_{k+1}}.
	\end{equation}
	To bound $E_2$, we perform a union bound over all $n$ coordinates, each of which fails to update with probability $e^{-T_\delta}$. By \eqref{eq:coupon} then,
    
	\begin{equation}\label{eq:E2}
		\P(E_2^c) \le ne^{-T_\delta} \lesssim 1/n^4.
	\end{equation}
	Also, 
	\begin{equation}\label{eq:E3}
		\P(E_3^c) \lesssim 1/n^2
	\end{equation}
	by \eqref{eq:Ek}. 
    Thus, by \eqref{eq:p'k}--\eqref{eq:E3},

	\begin{align}\label{eq:labmdprime}
		\nu' \ge \sum_{k=0}^{N-1}p'_k \gtrsim T\sum_{k=0}^{N-1}\mu_{y_{k+1}} - O(N/n^4) = T\sum_{k=0}^{N}\mu_{y_{k}} - o(e^{-x_0}),
	\end{align}
    where for the last equality we used that $T\mu_{y_0} = o(e^{-x_0})$ by \eqref{eq:mu-y0-e-x0}. Lower bounding the Riemann sum,
	\begin{align}\label{eq:riemann2}
		T\sum_{k=0}^{N}\mu_{y_{k}} \ge Tf_n \log^3 \! n\int_{y_0}^{y_N} \mu_y \, dy 
			&= Tf_n \log^3 \! n \bigg(\frac{\mu_{y_0}}{y_0} - \frac{\mu_{y_N}}{y_N}\bigg) \nonumber \\
            &\sim e^{-x_0}\big(1 - \frac{y_0\mu_{y_N}}{y_N\mu_{y_0}}\big).
	\end{align}
    To remove the final term, note that, as both $y_N - y_0$ and $y_N + y_0$ are large, bounded below by $y_0 = \Theta(f_n)$, we have
    \begin{align}\label{eq:remove-yN}
        \frac{y_0\mu_{y_N}}{y_N\mu_{y_0}} = \exp(-(y_N^2 - y_0^2)/2) = \exp(-(y_N - y_0)(y_N + y_0)/2) \to 0.
    \end{align}
    By \eqref{eq:labmdprime}--\eqref{eq:remove-yN}then, we have $\nu' \gtrsim e^{-x_0}$, the lower bound of \eqref{eq:lambd}. 
    
    This completes the proof of \eqref{eq:in-distr}.

\subsection{Variance} \label{sec:variance} We now prove \eqref{eq:var}, that $\var(G_n) \to \var(G) = \pi^2/6$. Having shown \eqref{eq:in-distr}, stating that $G_n \xrightarrow[]{d} G$, it suffices to demonstrate convergence of the second moments $\E[G_n^2] \to \E[G^2]$. We use the identities
\begin{align*}
    \frac{1}{2} \, \E[G^2] = \int_0^\infty x \, \P(|G| > x) \, dx &= \int_0^\infty x \, \P(G > x) \, dx + \int_{-\infty}^0|x|\, \P(G \le x)\, dx \\
    &=: R + L
\end{align*}
and, similarly,
\begin{align*}
    \frac{1}{2} \, \E[G_n^2] &= \int_0^{3 \log n} x \, \P(G_n > x) \, dx + \int_{3 \log n}^\infty x \, \P(G_n > x) \, dx \\
    & \qquad + \int_{- 2 \log \log n^3}^0 |x| \, \P(G_n \le x) \, dx +\int_{-\infty}^{- 2 \log \log n^3} |x| \, \P(G_n \le x) \, dx \\
    &=: (R_{n,1} + R_{n,2}) + (L_{n,1} + L_{n,2}).
\end{align*}
We will show via the dominated convergence theorem that $R_{n,1}, L_{n,1} \to R,L$, and we will also show that $R_{n,2},L_{n,2}$ vanish.

For the right tails, take $0 \le x_0 \le 3 \log n$. We have $\P(G_n > x_0) = \P(X > 0)$ which by Markov's inequality is at most $\E[X] = \nu$. By \eqref{eq:lambd} then, 
\begin{equation}\label{eq:right-tail}
    x_0 \, \P(G_n > x_0) < 2x_0e^{-x_0},
\end{equation}
assuming $n$ is sufficiently large. Since $2xe^{-x}\mathbf{1}_{\{x \ge 0\}}$ is integrable, and $x\P(G_n > x) \to x\P(G > x)$ pointwise by \eqref{eq:in-distr}, indeed we have $R_{n,1} \to R$ by the dominated convergence theorem. 

For $R_{n,2}$, we have that $G_n$ is at most order $\sqrt{n} \, f_n = o(n)$, as $S_n$ lies between $0$ and $n$. For large $n$ then, using \eqref{eq:right-tail} again, we have
$$
    R_{n,2} \le o(n^2) \, \P(G_n > 3 \log n) \le o(n^2)2n^{-3} \to 0.
$$

For the left tails, take $- 2 \log \log n^3 \le x_0 \le 0$. We have $\P(G_n \le x_0) \le \P(X' = 0)$, which, as the $X'_k$ are independent, is at most $e^{-\nu'}$. Thus, for large enough $n$ we have 
\begin{equation}\label{eq:right-tail2}
    |x_0|\P(G_n \le x_0) \le |x_0|\exp(-e^{-x_0/2}).
\end{equation}
Similarly, as $xe^{-x/2}\mathbf{1}_{\{x \le 0\}}$ is integrable and $|x|\P(G_n \le x) \to |x|\P(G \le x)$ pointwise, we have $L_{n,1} \to L$. And
\begin{align*}
    L_{n,2} \le o(n^2) \, \P(G_n \le - 2 \log \log n^3) \le o(n^2) \, \frac{1}{n^3}
\end{align*}
which vanishes, completing the proof.

\qed

\section{proof of \thref{lem:max-tail}} \label{sec:m}

Recall that \thref{lem:max-tail} states that  for $x=x_n$ and $T=T_n$ satisfying  $x = \Omega(\sqrt{\log n})$, $x = o(n^{1/6})$, and $\log n \ll T \ll 1/\mu_x$ we have $\P(m(T) \ge x) \sim \mu_x \,T$
    with 
$$
    \mu_x = \frac{x\exp(-x^2/2)}{\sqrt{2 \pi}} \sim x^2\Tail(x).
$$

Denote the \emph{exceedance time} above $x \in \R$ on $[t_1, t_2]$ (the values $t_1, t_2 \in [0,\infty)$ are simply scalars and are unrelated to the $t_k$ in the proof of \thref{thm:gumby}) by
$$\ExcTime(t_1, t_2) = \ExcTime(t_1, t_2, x) := \int_{t_1}^{t_2} \mathbf{1}_{\{s(t) \ge x\}} \, dt,$$
and use the convention $\ExcTime(t) := \ExcTime(0, t)$. We note that for any stopping time $\tau \ge 0$ and $t \ge 0$, letting $\E_y$ denote the expectation conditioned on $s(0) = y$,
\begin{equation}\label{eq:exc-from-stopping-time}
    \E[\ExcTime(\tau, \tau + t) \mid (s(t))_{0 \le t \le \tau}] = \E_{s(\tau)} \, \ExcTime(t) \quad a.s.
\end{equation}
Also, for any $z \ge y \in \R$, 
\begin{equation}\label{eq:exc-monotonic}
    \E_z \, \ExcTime(t) \ge \E_y \, \ExcTime(t).
\end{equation}
\eqref{eq:exc-monotonic} may be proved by a coupling where, once $s(0)$ is set, the time and value of all coordinate updates is shared, so that $s(t)$ is monotonic in $s(0)$ (and after time $\tau(0)$, $s(0)$ is forgotten).

Denote the maximum possible value that $s(t)$ may take by $x_M := n(1-p)/a_n$. For $0 \le x \le x_M$, let $x' \ge x$ be minimal so that $s(t)$ is supported on $x'$, that is $xa_n + pn$ is an integer. Recall that $\tau(t)$ is the time by which each coordinate has updated its value from time $t$. For convenience, let $\tau^+(t) := \max(\tau(t), t + 2 \log n)$, and let $\tau^+ := \tau^+(0)$. It is easy to show that $\E \tau^+ \sim 2 \log n$. This is because at any time $t \ge 0$, each coordinate fails to update by time $t$ with probability $e^{-t}$, so by a union bound over the $n$ coordinates,
\begin{equation}\label{eq:coupon-ub}
    \E \tau^+ - 2 \log n = \int_{2 \log n}^\infty \P(\tau(0) \ge t) \, dt \le  \int_{2 \log n}^\infty ne^{-t} \, dt = 1/n.
\end{equation}

We will show that $\E_{x'} \, \ExcTime(\log n) \sim \E_{x'} \, \ExcTime(\tau^+) \sim x^{-2}$:

\begin{lemma}\thlabel{lem:cluster}
    For $x$ satisfying $x = \Omega(\sqrt{\log n})$ and $x = o(n^{1/6})$, 
    $$
        \E_{x'} \, \ExcTime(\log n) \sim x^{-2} \sim \E_{x'} \, \ExcTime(\tau^+).
    $$
\end{lemma}

First, we use \thref{lem:cluster} to carry out the proof of \thref{lem:max-tail}.
\begin{proof}[Proof of \thref{lem:max-tail}]

    Note by stationarity that 
    \begin{equation}\label{eq:stationarity}
        \E \, \ExcTime(T) = T \, \P(s(0) \ge x) \sim T\Tail(x) \sim x^{-2}T\mu_x.
    \end{equation}
    The similitude with the Gaussian tail comes from \cite[Theorem 2]{cramer2018new} and the requirement that $x = o(n^{1/6})$. 

    For the upper bound on $p_{x,T} := \P(m(T) \ge x)$, let $\tau_x$ be the first exceedance time $t \ge 0$ where $s(t) \ge x$, so that $m(T) \ge x$ if and only if $\tau_x \le T$. We claim that 
    \begin{equation*}\label{eq:uniform}
        (\tau_x \mid \tau_x \le T) \preceq_{} \operatorname{Unif}(0,T).    
    \end{equation*}
    Indeed, on $m(T) \ge x$, let $t^* \in [0,T]$ be a uniformly chosen exceedance time where $s(t^*) \ge x$. Then $t^* \sim \operatorname{Unif}(0,T)$ and $t^* \ge \tau_x$. 
    
    Therefore, 
    \begin{align}\label{eq:lb-with-buffer}
        \E[\ExcTime(T) \mid m(T) \ge x] 
            &= \E[\ExcTime(T) \mid \tau_x \le T] \nonumber \\&\ge \frac{T - \log n}{T} \, \E[\ExcTime(T) \mid \tau_x \le T - \log n] \nonumber \\
            &\sim \E[\ExcTime(T) \mid \tau_x \le T - \log n].
    \end{align}
    
    Note that on $\tau_x \le T - \log n$, we have $\ExcTime(T) \ge \ExcTime(\tau_x, \tau_x + \log n)$. And since $s(\tau_x) \ge x'$ (we have equality except when $\tau_x = 0$), by \eqref{eq:exc-from-stopping-time} and \eqref{eq:exc-monotonic} we have
    \begin{align}\label{eq:buffer}
        \E[\ExcTime(T) \mid \tau_x \le T - \log n] 
            & \overset{a.s.}{\ge} \E[\ExcTime(\tau_x, \tau_x + \log n) \mid \tau_x \le T - \log n] \nonumber \\
            &\overset{a.s.}{\ge} \E_{x'}\,  \ExcTime(\log n) \nonumber \\
            &\sim x^{-2},
    \end{align}
    invoking \thref{lem:cluster} for the last step.
    Combining \eqref{eq:lb-with-buffer} and \eqref{eq:buffer} then, we have
    \begin{align*}
        \E \, \ExcTime(T) = p_{T,x}\, \E[\ExcTime(T) \mid m(T) \ge x] \gtrsim p_{T,x} \, x^{-2}.
    \end{align*}
    Therefore, by \eqref{eq:stationarity}, 
    \begin{equation}\label{eq:main-lemma-ub}
        p_{T,x} \lesssim x^2 \, \E \, \ExcTime(T) \sim T\mu_x.
    \end{equation}
    This establishes half of the lemma, and also that $p_{T,x} = o(1)$.

    For the lower bound, to analyze $\ExcTime(T)$, we go up to the first exceedance time $\tau_x$, allow the process to regenerate, and note that now there is less than $T$ time left. In symbols, 
    $$\ExcTime(T) \le \ExcTime\big(\tau_x, \tau^+(\tau_x)\big) + \ExcTime\big(\tau^+(\tau_x), \tau^+(\tau_x) + T\big) =: \ExcTime_1 + \ExcTime_2.$$ 
    Note that $\ExcTime_2 \mid \tau^+(\tau_x) \overset{d}{=} \ExcTime(T)$. Therefore, since $\ExcTime(T) = 0$ whenever $\tau_x > T$,
    \begin{align*}
        \E \, \ExcTime(T)
            &= \E \mathbf{1}_{\{\tau_x \le  T\}}\ExcTime(T) \\
            &\le \E[\mathbf{1}_{\{\tau_x = 0\}}\ExcTime_1] + \E[\mathbf{1}_{\{0 < \tau_x \le T\}}\ExcTime_1] + p_{T,x} \, \E[\ExcTime_2 \mid \tau_x \le T]\\
            &= \E[\mathbf{1}_{\{\tau_x = 0\}}\ExcTime(\tau^+)] + \P(0 < \tau_x \le T)\E_{x'}\, \ExcTime(\tau(0)) + p_{T,x}\, \E \, \ExcTime(T)\\
            &\le \E \, \ExcTime(\tau^+) + p_{T,x} \, \E_{x'}\, \ExcTime(\tau^+) + o(\E \, \ExcTime(T)).
    \end{align*}
    Now by stationarity, the first term $\E \, \ExcTime(\tau^+)$ (which does not start from $x'$ as in \thref{lem:cluster}) is exactly $\E[\tau^+] \, \P(s(0) \ge x) \sim 2\log n \, \Tail(x) \ll T \Tail(x) \sim \E\, \ExcTime(T)$, using \eqref{eq:stationarity}. Rearranging then, we have 
    \[
        p_{T,x} \gtrsim \frac{\E \, \ExcTime(T)}{\E_{x'} \, L(\tau^+)}.
    \]
    Again then applying \thref{lem:cluster} and \eqref{eq:stationarity}, we have  $p_{T,x} \gtrsim T \mu_x$, 
    establishing the second half of the lemma.
\end{proof}

Finally, we prove \thref{lem:cluster}. 
\begin{proof}[Proof of \thref{lem:cluster}]
    The bulk of both exceedance quantities comes from times on the order of $x^{-2}$. Let $Q(t) := \P_{x'}(s(t) \ge x')$ where $\P_{x'}$ is the probability measure conditioned on $s(0) = x'$. We use the Berry-Esseen theorem to approximate $Q(t)$.

    Let $\Vec{X}(t) = (X_i(t))_{1 \le i \le n}$ denote the value of each coordinate at time $t$, so that 
$$
    S(t) = \sum_{i = 1}^n X_i(t). 
$$
Each coordinate updates by time $t$ independently with probability $1 - e^{-t}$, with each updated value independent. Therefore, conditioned on $\Vec{X}(0)$, $\Vec{X}(t)$ is composed of independent (non-identical) Bernoulli random variables, where $\E[X_i(t) \mid \vec{X}(0)] = q_{X_i(0)}$ for 
\begin{equation}\label{eq:bern-coord}
    q_X := (1 - e^{-t})p + e^{-t}X = p + e^{-t}(X - p).
\end{equation}
Thus $S(t)$ is the independent sum of $S(0)$ copies of $\Ber(q_1)$ and $n - S(0)$ copies of $\Ber(q_0)$ random variables, and as expected only depends on $\Vec{X}(0)$ through $S(0)$. Conditioning on $s(0) = x'$ then, i.e., 
$
    S(0) = pn + x'a_n,
$
we get
\begin{equation}\label{eq:clt-mean}
    \mu(t) := \E_{x'}S(t)
        = S(0)q_1 + (n - S(0))q_0 = 
        pn + e^{-t}x'a_n.
\end{equation}

For the variance, letting $\sigma^2_X := q_X(1 - q_X)$ and recalling that $a_n = \sigma \sqrt{n} = \sqrt{p(1-p)n}$,
\begin{align}
\sigma^2(t) := \var_{x'} \! S(t)
&= S(0)\sigma^2_1 + (n-S(0))\sigma^2_0 \nonumber \\
&= (pn+x'a_n)q_1(1 - p)(1-e^{-t}) \nonumber \\
&\qquad + (n(1-p)-x'a_n)p(1-e^{-t})(1 - q_0) \nonumber \\
&= \big[1 - e^{-t}\big]\big[pn(1-p)(e^{-t} + 1) + x'a_ne^{-t}(1-2p)\big] \nonumber \\
&= \big[1 - e^{-t}\big]\big[a_n^2(e^{-t} + 1) + O(x'a_n)\big] \nonumber \\
&= a_n^2(1 - e^{-t})(1 + e^{-t})(1 + O(x'/a_n)).
\end{align}

By the Berry-Esseen theorem, and using that Bernoulli random variables are bounded, we have, uniformly over all $t > 0$,
\begin{align}\label{eq:Qt}
    Q(t) = \P_{x'}(S(t) \ge pn + x'a_n) 
        &= \Tail\bigg(\frac{pn + x'a_n - \mu(t)}{\sigma(t)}\bigg) + O\bigg(\frac{1}{\sigma(t)}\bigg) \nonumber \\
        &= \Tail\bigg(x'\sqrt{\frac{1 - e^{-t}}{1 + e^{-t}}(1 + O(x'/a_n)}\bigg)  + O\bigg(\frac{1}{\sqrt{n(1 - e^{-t})}}\bigg).
\end{align}

Taylor expanding $e^{-t}$ for small $t$ gives $e^{-t} = 1 - t + O(t^2)$. Therefore, if we restrict to  $t \le x^{-3/2}$, we have 
\begin{align}
    \Tail\bigg(x'\sqrt{\frac{1 - e^{-t}}{1 + e^{-t}}\big(1 + O(x'/a_n)\big)}\bigg) 
        &= \Tail\bigg(x'\sqrt{\frac{t}{2}\big(1 + O(t) + O(x'/a_n)\big)}\bigg) \nonumber \\
        &\sim \frac{\exp\big(-\frac{1}{2}(x')^2[t/2 + O(t^2) + O(x'/a_n)]\big)}{\sqrt{2\pi} \, x' \sqrt{t/2}} \nonumber \\
        &\sim \frac{\exp\big(-\frac{1}{2} (x')^2 \, t/2\big)}{\sqrt{2\pi} \, x' \, \sqrt{t/2}} \label{eq:Qt2-midway} \\
        &\sim \Tail\big(x'\sqrt{t/2}\big). \label{eq:Qt2}
\end{align}

Equation \eqref{eq:Qt2-midway} follows from the fact that both $(x')^2t^2$ and $(x')^3/a_n$ are small under the assumptions that $t \le x^{-3/2}$ and $1 \ll x \ll n^{1/6}$, and since $x' \sim x$. Also, in this regime, 
\begin{equation}\label{eq:error-small-regime}
    \frac{1}{\sqrt{n(1 - e^{-t})}} = O\bigg(\frac{1}{\sqrt{nt}}\bigg).
\end{equation}

Note that both terms of \eqref{eq:Qt} are decreasing in $t$, and that $\log n = O(x^2)$ by the assumption on $x$. Therefore, combining \eqref{eq:Qt}, \eqref{eq:Qt2}, and \eqref{eq:error-small-regime}, 
\begin{align}\label{eq:bulk}
    \E_{x'} \, \ExcTime(x^{-3/2}, 2 \log n)  
        &= \int_{x^{-3/2}}^{2 \log n}Q(t)\, dt \nonumber \\\
        &\le 2 \log n\bigg[\Tail\bigg(x'\sqrt{x^{-3/2}/2}\bigg) + O\bigg(\frac{1}{\sqrt{nx^{-3/2}}}\bigg)\bigg] \nonumber \\
        &= O(x^2)\Tail(\Theta(x^{1/4})) + O\bigg(\frac{x^{3/4} \log n}{\sqrt{n}}\bigg) \nonumber \\
        &= O(x^2)\exp(-\Omega(x^{1/2})) + O\bigg(\frac{\log n}{n^{3/8}}\bigg) \nonumber \\
        &= o(x^{-2}) + o\bigg(\frac{1}{n^{1/3}}\bigg) \nonumber \\
        &= o(x^{-2}).
\end{align}

We claim it suffices to check exceedances up to time $x^{-3/2}$, that is to show 
\begin{equation}\label{eq:et-bulk}
    \E_{x'} \, \ExcTime(x^{-3/2}) \sim x^{-2}.
\end{equation}
Indeed then, combined with \eqref{eq:bulk} and \eqref{eq:coupon-ub} we have that
\begin{align*}
    x^{-2} \sim \E_{x'} \, \ExcTime(x^{-3/2}) \le E_{x'} \, \ExcTime(\log n) 
        &\le \E_{x'} \, \ExcTime(\tau^+) \\ 
        &\le \E_{x'} \, \ExcTime(2 \log n) + \E[\tau^+ - 2 \log n]\\
        &\le \E_{x'} \, \ExcTime(x^{-3/2}) + \E_{x'} \, \ExcTime(x^{-3/2}, 2 \log n) + 1/n\\
        &\sim x^{-2},
\end{align*}
where at the end we used that $n \gg x^6 \gg x^2$.

Thus it remains to show \eqref{eq:et-bulk}. We use \eqref{eq:Qt}, \eqref{eq:Qt2}, \eqref{eq:error-small-regime}, and the fact that $\int_{0}^\infty u \Tail(u) \, du = 1/4$:
\begin{align*}
    \E_{x'} \, \ExcTime(x^{-3/2}) 
        =  \int_{0}^{x^{-3/2}}Q(t)\, dt  &\sim \int_{0}^{x^{-3/2}}\Tail(x' \sqrt{t / 2})\, dt + O\bigg(\int_0^{x^{-3/2}} \frac{1}{\sqrt{nt}} \, dt \bigg)\nonumber \\
        &= \frac{4}{(x')^2}\int_0^{x'\sqrt{x^{-3/2}/2}} u\Tail(u) \, du + o\big( 1/\sqrt{n} \big) \nonumber \\
        &\sim x^{-2} + o(x^{-2}).
\end{align*}
For the first term in the last line, we used that $x'\sqrt{x^{-3/2}/2} \sim x^{1/4}/\sqrt{2} \gg 1$. For the second term, we used that $\sqrt{n} \gg x^3$. This demonstrates \eqref{eq:et-bulk} and completes the proof of the lemma.

\end{proof}

\section{Notation} \label{sec:notation}
Here we define our asymptotic and random variable notation.
\subsection{Asymptotic notation}
Take $f,g,h \colon \N \to (0, \infty)$. 
\begin{itemize}
    \item $f \ll g$ and $f = o(g)$ mean that $\lim f / g = 0$.
    \item $f = O(g)$ if $\limsup f/ g < \infty$. 
    \item $f = \Theta(g)$ if $f = O(g)$ and $g=O(f)$. 
    \item $f=\Omega(g)$ if $g = O(f)$. 
    \item $f \sim g$ and $f = g(1+o(1))$ both mean that $\lim f/g = 1$
    \item $f \lesssim g$ means that $\limsup f/g \leq 1 $. 
    \item $f \lesssim g + o(h)$ means that $\limsup f /(g + \epsilon h) \leq 1$ for all $\epsilon >0$.
    \item $f \lesssim g + O(h)$ means that $\limsup f / (g + C h) \leq 1$ for some $C>0$.   
    \item $f \sim g + o(h)$ means that $f \lesssim g + o(h)$ and $g \lesssim f + o(h)$. Similarly for $f \sim g + O(h)$. 
    \item $f \approx g$ means that $f$ is close to $g$ in an intuitive, but not rigorously stated sense. 
    \item A sequence of events occurs {with high probability} (w.h.p.) if their probabilities tend to one. 
\end{itemize}
\subsection{Random variables}
\begin{itemize}
    \item $\Bin(n,p)$ is the binomial distribution with parameters $n$ and $p$.
    \item  $\Geo(s)$ is the geometric distribution with mean $1/s$ supported on the positive integers. 
    \item $\Pois(\nu)$ is the Poisson distribution with mean $\nu$. 
    \item Bernoulli($s$) is the Bernoulli distribution with mean $s$. 
    \item $\Exp(\nu)$ is the exponential distribution with mean $1/\nu$. %
    \item $X \sim \Bin(n,p)$ means that the random variable $X$ is $\Bin(n,p)$-distributed and similarly for the other named distributions.
    \item $X \overset{d} = Y$ means that $X$ and $Y$ have the same distribution. 
    \item $X \preceq_{}Y$ if $\P(X \geq a) \leq \P(Y \geq a)$ for all $a \in \mathbb R$. 
    \item $X \overset{a.s.} \leq Y$ denotes an inequality that holds almost surely. 
    \item $X_n \overset{d} \to X$ if $\P(X_n \leq x) \to \P(X \leq x)$ for all continuity points of $\P(X \leq x)$. 
    \item $X_n \overset{\P}\to x$ if $|X_n - x| \leq \epsilon$ w.h.p.\ for all $\epsilon >0$. 
\end{itemize}

\section*{Acknowledgements}
Junge was partially supported by NSF DMS Grant 2238272. Kaufman was partially supported by NSF DMS Grants 2238272 and 2349366. Part of this research was conducted during the 2025 Baruch College Discrete Mathematics NSF Site REU.

\bibliographystyle{amsalpha}
\bibliography{BA.bib}

\end{document}